\documentclass[a4paper,10pt]{amsart}
\usepackage[utf8]{inputenc}
\usepackage{etex}
\usepackage{lmodern}
\usepackage[english]{babel}

\usepackage{amsmath,amssymb,mathrsfs,amsfonts,dsfont,amsthm,mathtools}

\usepackage[pdfstartview=FitH,breaklinks=true,bookmarksopen=true]{hyperref}
\hypersetup{colorlinks=true,citecolor=blue,linkcolor=blue}
\usepackage{bookmark}
\usepackage{amsrefs}

\newtheorem{theorem}{Theorem}[section]
\newtheorem{corollary}[theorem]{Corollary}

\newtheorem{lemma}[theorem]{Lemma}
\newtheorem{proposition}[theorem]{Proposition}

\theoremstyle{definition}

\newcommand{\Sphere}{\ensuremath{\mathbb{S}}^2}
\newcommand{\Cset}{\ensuremath{\mathbb C}}
\newcommand{\Rset}{\ensuremath{\mathbb R}}
\newcommand{\Zset}{\ensuremath{\mathbb Z}}
\newcommand{\Nset}{\ensuremath{\mathbb N}}
\newcommand{\dd}{\ensuremath{\mathrm d}}
\newcommand{\ee}{\ensuremath{\mathrm e}}
\newcommand{\ii}{\ensuremath{\mathrm i}}

\newcommand{\Ein}{\ensuremath{E^\mathrm{i}}}
\newcommand{\Esca}{\ensuremath{E^\mathrm{s}}}
\newcommand{\Hin}{\ensuremath{H^\mathrm{i}}}
\newcommand{\Hsca}{\ensuremath{H^\mathrm{s}}}

\newcommand{\nad}{\ensuremath{n_\alpha^\delta}}
\newcommand{\ntrue}{\ensuremath{n^\dagger}}

\newcommand{\solsetb}{\ensuremath{\mathfrak D}}
\newcommand{\Xspace}{\ensuremath{\mathcal X}}
\newcommand{\Yspace}{\ensuremath{\mathcal Y}}
\newcommand{\Hmdom}{\ensuremath{H^m(C(\pi))}}
\newcommand{\Hm}{\ensuremath{H^m}}

\newcommand{\Rp}{\ensuremath{R^\prime}}
\newcommand{\Rpp}{\ensuremath{R^{\prime\prime}}}
\newcommand{\dmax}{\ensuremath{\delta_\mathrm{max}}}
\newcommand{\dmaxtilde}{\ensuremath{\widetilde{\delta}_\mathrm{max}}}

\newcommand{\tmin}{t_0}
\newcommand{\Me}{\ensuremath{L_{m}}}
\newcommand{\Mgos}{\ensuremath{M_1}}
\newcommand{\Mdata}{\ensuremath{M_{2}}}
\newcommand{\Mdiff}{\ensuremath{M_{3}}}
\newcommand{\Mms}{\ensuremath{M_{4}}}
\newcommand{\low}{\ensuremath{1/b}}
\newcommand{\data}{\ensuremath{\norm{w_{1}-w_{2}}{(L^2(R\Sphere\times R\Sphere))^{3\times 3}}}}

\newcommand{\fourier}[1][n]{\ensuremath{\widehat{#1} (\gamma)}}
\newcommand{\fourierdiff}[2]{\ensuremath{\left(\widehat{#1}-\widehat{#2}\right)(\gamma)}}
\newcommand{\mult}{\ensuremath{\langle \gamma \rangle}}

\newcommand{\rbra}[1]{\ensuremath{\left( #1 \right)}}
\newcommand{\sbra}[1]{\ensuremath{\left[ #1 \right]}}
\newcommand{\cbra}[1]{\ensuremath{\left\{ #1 \right\}}}
\newcommand{\Tcbra}[1]{\ensuremath{\{ #1 \}}}
\newcommand{\absval}[1]{\ensuremath{\left\lvert #1 \right\rvert}}
\newcommand{\Tabsval}[1]{\ensuremath{\lvert #1 \rvert}}
\newcommand{\norm}[2]{\ensuremath{\left\lVert #1 \right\rVert_{#2}}}
\newcommand{\Tnorm}[2]{\ensuremath{\lVert #1 \rVert_{#2}}}
\newcommand{\pairing}[2]{\ensuremath{\left\langle #1,#2 \right\rangle}}
\newcommand{\diffq}[2]{\ensuremath{\frac{\partial #1}{\partial #2}}}

\newcommand{\curl}{\ensuremath{\nabla \times}}
\newcommand{\grad}{\ensuremath{\nabla}}
\newcommand{\Q}{\ensuremath{\mathcal Q}}

\DeclareMathOperator{\dom}{dom}

\DeclareMathOperator{\supp}{supp}
\DeclareMathOperator*{\argmin}{arg\,min}

\usepackage{amsaddr}
\title[vsc and stability of inverse em scattering]{Variational source conditions and stability estimates  for inverse electromagnetic  medium scattering problems}
\author{Frederic Weidling and Thorsten Hohage}
\address{Institute for Numerical and Applied Mathematics, University of Goettingen, Lotzestr. 16-18, 37083 Goettingen, Germany}
\email{f.weidling@math.uni-goettingen.de}

\begin{document}

\begin{abstract}
This paper is concerned with the inverse problem to recover the scalar, complex-valued 
refractive index of a medium from measurements of scattered time-harmonic electromagnetic waves at a fixed frequency. The main results are two variational 
source conditions for near and far field data, which imply logarithmic rates 
of convergence of regularization methods, in particular Tikhonov regularization, 
as the noise level tends to $0$. Moreover, these variational source conditions 
imply conditional stability estimates which improve and complement known 
stability estimates in the literature. 
\end{abstract}

\maketitle

\section{Introduction}
	In this paper we study the behavior of time-harmonic electromagnetic waves in an inhomogeneous isotropic medium. The propagation of the time-harmonic electric field $E$ is described by the equation
	\begin{equation}\label{eq:curlcurlE}
		-\curl \curl E+ \kappa^2 n E=0
	\end{equation}
	where the \emph{wave number} $\kappa$ and the \emph{refractive index} $n$ are given by
	\begin{equation}\label{eq:kappandef}
		\kappa:=\omega \sqrt{\epsilon_0 \mu_0},\qquad n(x):=\frac{1}{\epsilon_0} \rbra{\epsilon(x)+ \ii \frac{\sigma(x)}{\omega}}.
	\end{equation}
Here  $\epsilon$ and $\sigma$ denote the electric permittivity and the conductivity  of the medium, respectively, and $\omega$ denotes the frequency of the time-dependent electric field 
$\mathcal{E}(x,t) = \Re(\epsilon_0^{-1/2}E(x)\exp(-i\omega t))$. 
Supposing that the inhomogeneity is compactly supported, we can assume w.l.o.g.\ 
that $\supp(n-1)\subset B(R):=\Tcbra{x\in \Rset^3 \colon \Tabsval{x}<R}$ with 
$R=\pi$ by possibly rescaling $x$,
and sufficiently smooth a unique solution to \eqref{eq:curlcurlE} exist under a suitable radiation condition defining the behavior of the field at infinity as will be detailed in Section \ref{sec:problem}.
	
	Corresponding inverse problems are to probe the medium with incident fields $\Ein$ fulfilling \eqref{eq:curlcurlE} for $n\equiv 1$ and measure the corresponding scattered fields $\Esca=E-\Ein$ with the aim of recovering the refractive index, see Section \ref{sec:problem} for more details. The inverse problems will be formulated as operator equations of the form
	\begin{equation*}
		F(n)=y
	\end{equation*}
	with a mapping $F:\dom(F)\subset \Xspace\to \Yspace$ between Hilbert spaces 
	$\Xspace,\Yspace$. 
	Let $\ntrue$ denote the exact solution and $y^\delta\in \Yspace$ perturbed data satisfying  $\Tnorm{F(\ntrue)-y^\delta}{\Yspace}\leq \delta$. To find a stable approximation to $\ntrue$ from such data one needs to employ regularization techniques. A prominent example is Tikhonov regularization where the approximation $\nad$ is defined by
	\begin{equation}
		\nad \in \argmin_{n\in\dom(F)\cap\Xspace} \sbra{\frac{1}{\alpha} \norm{F(n)-y^\delta}{\Yspace}^2+\frac{1}{2}\norm{n}{\Xspace}^2}. \label{eq:tikhonov}\\
	\end{equation}
	Major issues of regularization theory are the convergence of $\norm{\ntrue-\nad}{\Xspace}\rightarrow 0$ as $\delta\rightarrow 0$ and the rate of this convergence for an appropriate parameter choice rule $\alpha=\alpha(\delta,y^\delta)$. To obtain such rates one needs additional assumptions on the true solution (see \cite[Prop 3.11]{Engl1996}). Starting with \cite{Hofmann2007} these assumptions are now frequently 
formulated as \emph{variational source conditions} 
	\begin{equation}
		\forall n \in \dom(F)\qquad \frac{\beta}{2}\norm{\ntrue-n}{\Xspace}^2\leq \frac{1}{2}\norm{n}{\Xspace}^2-\frac{1}{2} \norm{\ntrue}{\Xspace}^2 +\psi\rbra{\norm{F(n)-F(\ntrue)}{\Yspace}^2}, \label{eq:vsc}
	\end{equation}
	where $\beta\in(0,1]$ and $\psi\colon [0,\infty)\rightarrow[0,\infty)$ is an index function, that is $\psi$ is a continuous, monotonically increasing function satisfying $\psi(0)=0$. We only refer to \cite{Flemming:12,SKHK:12} for overviews. 
	
	Variational source conditions have many advantages over classical spectral source condition which we discussed in detail in \cite{Hohage2015}. 
Let us mention only two of them here: Firstly, if the index function $\psi$ is concave, it can be shown by a simple argument \cite{Grasmair2010} (see also \cite[Thm.\ 3.3]{Werner2012}) that they lead to the convergence rate
	\begin{equation}\label{eq:rate}
		\frac{\beta}{2}\norm{\ntrue-\nad}{\Xspace}^2\leq 4\psi(\delta^2), 
	\end{equation}
for Tikhonov regularization	with an optimal choice of the regularization parameter $\alpha$. 
Secondly, if a variational source condition with the some function $\psi$
holds for all $\ntrue\in \mathcal{K}$, where $\mathcal{K}$ is a (usually compact) subset of $\Xspace$, they imply the stability estimate
	\begin{equation}\label{eq:stability}
		\forall n_1,n_2\in \mathcal{K} \qquad \frac{\beta}{2}\norm{n_1-n_2}{\Xspace}^2\leq \psi\rbra{\norm{F(n_1)-F(n_2)}{\Yspace}^2}, 
	\end{equation}
	while it is not clear whether every stability estimate can be sharpened to a variational source condition.
	
	There are however only few verifications of variational source conditions so far. They can under certain assumptions be derived from spectral source conditions, but then they do not yield any new information. For linear operators $F$ and $l^q$ penalties with respect to certain bases in the range of $F^*$, characterizations of variational source condition have been derived in \cite{Anzengruber2013,Burger2013,Flemming2015}. Reformulations of \eqref{eq:vsc} have been proven for phase retrieval and an option pricing problem in \cite{Hofmann2007}. Moreover we showed that the acoustic inverse medium scattering problem fulfills such a condition recently \cite{Hohage2015}.
	
	The purpose of this paper is to demonstrate how similar techniques can be applied to the electromagnetic inverse medium scattering problem to derive a variational source condition and hence also a stability estimate via \eqref{eq:stability}. We are considering refractive indices fulfilling
	\begin{equation}\label{eq:defi_solset}
		n\in\solsetb:=\cbra{n\in C^{1,\alpha}(\Rset^3)\colon \supp(1-n)\subset B(\pi), \Re (n)\geq b, \Im(n)\geq0}, \quad b>0
	\end{equation}
	and an additional Sobolev smoothness on the set $C(\pi)=(-\pi,\pi)^3$. 
	
	As a first kind of measurement operator consider the \emph{near field operator} $F_\mathrm{n}$ mapping an refractive index to the corresponding Green's tensor $w_n(x,y)$ measured on $R\Sphere\times R\Sphere$ with the unit sphere 
	$\Sphere:=\Tcbra{x\in \Rset^3\colon \Tabsval{x}=1}$ and $R>\pi$.
	
	\begin{theorem}\label{thm:vscnear}
		Let $s>m>7/2$, $s\neq 2m+3/2$ and $R>\pi$. Suppose that the true refractive index $\ntrue$ satisfies $\ntrue\in\solsetb \cap H^s$ with $\Tnorm{\ntrue}{H^s}\leq C_s$. Then a variational source condition \eqref{eq:vsc} holds true for the operator $F_\mathrm{n}$ with $\dom(F_\mathrm{n}):= \solsetb \cap H^m(C(\pi))$, $\Yspace = (L^2(R\Sphere\times R\Sphere))^{3\times3}$, $\beta=1/2$, and $\psi$ given by 
		\begin{equation*}
			\psi_\mathrm{n}(t):=A\rbra{\ln(3+t^{-1})}^{-2\nu}, \qquad \nu:=\min \cbra{\frac{s-m}{m+5/2},\frac{s-m}{s-m+1}},
		\end{equation*}
		where the constant $A>0$ depends only on $m, s, C_s, \kappa, b$ and $R$.
	\end{theorem}
	As in \cite{Hohage2015} the choice of $\beta\in(0,1)$ is actually arbitrary, but $A$ depends on the choice of $\beta$. For a discussion of the exceptional case $s=2m+3/2$ see \cite[Remark 4.4]{Hohage2015}. Using the results leading to \eqref{eq:rate} and \eqref{eq:stability} we obtain the following corollary:
	
	\begin{corollary}\label{cor:near}
		Let $s>m>7/2$, $s\neq 2m+3/2$ and $R>\pi$. Assume that the refractive indices $\ntrue, n_1$ and $n_2$ satisfy $\ntrue, n_1, n_2\in \solsetb \cap H^s$ 
	with $\Tnorm{\ntrue}{H^s},\Tnorm{n_j}{H^s}\leq C_s$	for some $C_s\geq 0$, then 
		\begin{itemize}
			\item \emph{Convergence rate:} the error bound 
			\begin{eqnarray*}
				\norm{\nad-\ntrue}{\Hm} \leq 4\sqrt{A} \rbra{ \ln (3+\delta^{-2})}^{-\nu}
			\end{eqnarray*}
			in terms of the noise level $\delta$ with respect to the operator $F_\mathrm{n}$ holds true for the regularization scheme \eqref{eq:tikhonov} if 
$\alpha= (2A\frac{\partial \ln(3+t^{-1})^{-2\nu}}{\partial t}\big|_{t=4\delta^2})^{-1}$.
			\item \emph{Stability estimate:} one obtains the estimate
			\begin{equation*}\label{eq:stabnear}
				\norm{n_1-n_2}{\Hm}\leq 2\sqrt{A} \rbra{\ln \rbra{3+\norm{F_\mathrm{n}(n_1)-F_\mathrm{n}(n_2)}{(L^2(R\Sphere \times R\Sphere))^{3 \times 3}}^{-2}}}^{-\nu}.
			\end{equation*}
		\end{itemize}
	\end{corollary}
	
As $R\gg 1$ in many situations, in scattering theory one often considers 
the limit $R\to\infty$ in a sense explained in Section \ref{sec:problem}. 
Then incident fields are plane waves instead of point sources, and it is assumed 
that only the so-called far field patterns of the scattered fields can be 
measured. In this case we will show essentially the same results, but with slightly 
smaller exponents:	
	\begin{theorem}\label{thm:vscfar}
		Under the assumptions of Theorem \ref{thm:vscnear} and for all $0<\theta<1$ a variational source condition \eqref{eq:vsc} holds true for the operator $F_\mathrm{f}$ with $\dom(F_\mathrm{f}):= \solsetb\cap H^m(C(\pi))$, $\Yspace = (L^2(\Sphere\times \Sphere))^{3\times3}$, $\beta=1/2$, and $\psi$ given by 
		\begin{equation*}
			\psi_\mathrm{n}(t):=B\rbra{\ln(3+t^{-1})}^{-2\nu\theta}
		\end{equation*}
		where $\nu$ is given as in Theorem \ref{thm:vscnear} the constant $B>0$ depends only on $m, s, C_s, \kappa, b$, and $R$.
	\end{theorem}
	Again one obtains as a corollary results via \eqref{eq:rate} and \eqref{eq:stability}.
	\begin{corollary}\label{cor:far}
		Suppose the assumptions of Corollary \ref{cor:near} hold true, then 
		\begin{itemize}
			\item \emph{Convergence rate:}   the error bound
			\begin{equation*}
				\norm{\nad-\ntrue}{\Hm} \leq 4\sqrt{B} \rbra{ \ln (3+\delta^{-2})}^{-\nu\theta}
			\end{equation*}
			in terms of the noise level $\delta$ with respect to the operator $F_\mathrm{f}$ holds true for the regularization scheme \eqref{eq:tikhonov} if 
$\alpha= (2B\frac{\partial \ln(3+t^{-1})^{-2\nu}}{\partial t}\big|_{t=4\delta^2})^{-1}$.
			\item \emph{Stability estimate:} one obtains the estimate
			\begin{equation*}
				\norm{n_1-n_2}{\Hm}\leq 2\sqrt{B} \rbra{\ln\rbra{3+\norm{F_\mathrm{f}(n_1)-F_\mathrm{f}(n_2)}{(L^2(\Sphere \times \Sphere))^{3\times 3}}^{-2}}}^{-\nu\theta}.
		      \end{equation*}
		\end{itemize}
	\end{corollary}
	
	Compared to our results for the acoustic medium scattering problem \cite{Hohage2015}  we obtain slower rates of convergence. The reason for this will be discussed after Lemma \ref{lem:difftodata}. However, it can be shown by entropy techniques 
	\cite{Mandache2001,CR:03} that the logarithmic rates of convergence are optimal up to the value of the exponent $\nu$.
	
	While we prove the first variational source condition for this problem, stability estimate exists for similar problems. Table \ref{tab:compare} gives a short overview over the different results known to the authors. For the problem with measured far field data as described by \eqref{eq:farfieldoperator} a stability estimate was proven in \cite{Haehner2000} under the assumption that $n_j\in C^{2,\gamma}$ for a $\gamma>0$. The stability estimate is in the $L^\infty$-norm and the obtained exponent of the logarithmic factor is $\nu=1/15$, but a very strong norm has to be applied in the image space.
	
	Furthermore there are stability estimates using Cauchy data \cite{Caro2010,Lai2015}. The Cauchy data in this setup is defined as the set of the tangential parts of the electric and magnetic field on a sphere $R\Sphere$ for all possible solutions of the time-harmonic Maxwell system in the ball $B(R)$ and the distance between different Cauchy sets is measured by a Hausdorff like distance. The stability estimate in \cite{Caro2010} is in the $H^1$ norm and holds also for non constant magnetic permeability $\mu$ but the obtained exponent is unknown and bounded by $1/3$ under similar assumptions as in our case.
	
The result in \cite{Lai2015} holds true only for small conductivities and 
constant electric permittivity. However it is the first H\"older-logarithmic  
stability estimate for electromagnetic scattering. This means that the stability 
estimate depends in an explicit way on the wave length $\kappa$ and for 
$\kappa\rightarrow \infty$ the logarithmic stability turns into a H\"older stability. 
The obtained exponent is linearly increasing with the assumed smoothness up to $\nu=1$.

	\begin{table}[ht]
		\centering
		\small
		\begin{tabular}{|p{0.15\textwidth}|p{0.16\textwidth}|p{0.18\textwidth}|p{0.16\textwidth}|p{0.17\textwidth}|}
		\hline
			& \bf new &\bf H\"ahner \cite{Haehner2000} &\bf Caro \cite{Caro2010} &\bf Lai et al \cite{Lai2015} \\ \hline
			\bf data & near/far field & far field & Cauchy & Cauchy\\ \hline
			\bf validity & global & local anywhere& global & local around $0$ \\ \hline
			\bf stability of & $\sigma, \epsilon$ & $\sigma, \epsilon$ & $\sigma, \epsilon, \mu$ & $\sigma$ \\ \hline
			\bf norm & $H^m$ & $L^\infty$ & $H^1$ & $H^{-s}$ \\ \hline
			\bf exponent & $<1$ & $1/15$ & unknown,$<\frac{1}{3}$ &  $\leq1$\\ \hline
			\bf special & &strong norm in image space& & H\"older-logarithmic \\ \hline
		\end{tabular}
		\caption{Comparison between different stability estimates.}\label{tab:compare}
	\end{table}

	In the following section we will give a precise formulation of the considered problems.  In Section \ref{sec:lowfourier} we discuss how to use complex geometrical optics solutions  to estimate differences of Fourier coefficients of refractive indices. The proof of Theorem \ref{thm:vscnear} is given in Section \ref{sec:proofnear}. The connection of near and far field data is discussed in Section \ref{sec:neartofar} which will provide an easy way to prove Theorem \ref{thm:vscfar} in Section \ref{sec:prooffar}. The construction of the complex geometric optics solution will be detailed in Appendix \ref{sec:GOS}.
	
\section{The direct and inverse problems}\label{sec:problem}

	In this paper we study  electromagnetic waves $(\mathcal E, \mathcal H)(x,t)$ in an inhomogeneous isotropic medium. The propagation of the electric field $\mathcal E$ and the magnetic field $\mathcal H$  is described by Maxwell's equations
	\begin{equation*}
		\curl \mathcal{E}+ \mu \diffq{\mathcal H}{t}=0, \qquad \curl \mathcal H - \epsilon \diffq{\mathcal E}{t}=J.
	\end{equation*}
	We will assume that the magnetic permeability $\mu=\mu_0$ is constant, which is a good approximation for most materials, while the electric permittivity $\epsilon=\epsilon(x)>0$ is allowed to vary in space. By Ohm's law the current density is given by $J(x,t)=\sigma(x) \mathcal E(x,t)$ where the conductivity satisfies $\sigma(x)\geq 0$. We will consider the case of a compactly supported inhomogeneity, that is $\epsilon(x)=\epsilon_0$ and $\sigma(x)=0$ for all $x$ larger then a certain radius $R$ and assume without loss of generality that $R=\pi$.
	
	Supposing a time-harmonic dependence of the electromagnetic wave of the form
	\begin{equation*}
		\mathcal E(x,t)=\Re\rbra{\frac{1}{\sqrt{\epsilon_0}} E(x) \ee^{-\ii \omega t}}, \qquad \mathcal H(x,t)=\Re\rbra{\frac{1}{\sqrt{\mu_0}} H(x) \ee^{-\ii \omega t}}
	\end{equation*}
	with frequency $\omega>0$ the fields $E$ and $H$ must satisfy the equations
\begin{subequations}\label{eqs:maxwell}
	\begin{equation}\label{eq:permaxwell}
		\begin{aligned}
			\curl E- \ii \kappa H =0\\
			\curl H+\ii \kappa n E =0
		\end{aligned}
	\end{equation}
	where the wave number $\kappa$ and the refractive index $n$ are given as in \eqref{eq:kappandef}. 
	In the following we will assume that the total field $(E,H)$ is given as the superposition of an incident field $(\Ein,\Hin)$ that solves the Maxwell equations \eqref{eq:permaxwell} for $n\equiv 1$ and a scattered field 
	\begin{equation}
	(\Esca,\Hsca)=(E,H)-(\Ein,\Hin).
	\end{equation}
	To guarantee that the scattered field is unique and models outgoing waves we assume that  the Silver-M\"uller radiation condition
	\begin{equation}
		\lim_{\absval{x}\rightarrow \infty} \rbra{ \Hsca(x) \times x - \absval{x} \Esca(x)}=0 \label{eq:silver}
	\end{equation}
	\end{subequations}
	is satisfied.	One can take the first equation in \eqref{eq:permaxwell}
	as an equation defining $H$ and thus obtain \eqref{eq:curlcurlE}. 
	
	
For a known incident wave $\Ein$ and a known refractive index $n$ in the admissible 
set $\solsetb$ defined in \eqref{eq:defi_solset} 
	the system \eqref{eqs:maxwell} has a unique solution $E\in C^1(\Rset^3)$.  
	$E$ is the unique solution of the electromagnetic Lippmann-Schwinger equation
	\begin{multline*}
		E(x)= \Ein(x) - \kappa^2 \int_{B(\pi)} \Phi_\kappa(x-y) \rbra{1-n(y)} E(y)\, \dd y\\ + \grad \int_{B(\pi)} \Phi_\kappa(x-y) \frac{1}{n(y)} \grad n(y) \cdot E(y) \,\dd y, 
	\end{multline*}
	and $H=(i\kappa)^{-1}\curl E$ where $\Phi_\kappa(x)=\ee^{\ii \kappa \absval{x}}/(4\pi \Tabsval{x})$ is the fundamental solution to the Helmholtz equation, see \cite{Colton2013}. The regularity assumptions on $n$ can be relaxed if those on the regularity of the total field are relaxed. 

	For our regularization scheme \eqref{eq:tikhonov} we will choose $\Xspace=\Hmdom$ for $m>7/2$ with the norm
	\begin{equation*}
		\norm{f}{\Hm}^2 = \sum_{\gamma \in \Zset^3} (1+\Tabsval{\gamma}^2)^m \absval{\fourier[f]}^2
	\end{equation*}
	where $\fourier[f]$ are the Fourier coefficients of $f$ in $C(\pi)$. 
	Our choice of $m$ implies that there exists some constant $\Me$ such that
	\begin{equation}\label{eq:defi_Mm}
		\Me \norm{f}{\Hmdom}\geq \norm{f}{C^2(C(\pi))}:= \max_{\absval \alpha \leq 2} \sup_{x\in C(\pi)} \absval{\partial_\alpha f(x)},
	\end{equation}
	Our requirements on the smoothness of $n$ and the additional lower bound $b$ on the electric permittivity will be needed in our analysis of the problem and are requirements for the construction of our main tool the complex geometric optics solutions, which we will present in detail in the appendix.
	
	The two inverse problems we consider differ in the type of incident fields and the type of measurements of the solution to \eqref{eqs:maxwell}. For the first inverse problem we consider incident fields generated by a time harmonic electromagnetic dipole with moment $a\in \Rset^3$ located at $y\in\Rset^3$:
	\begin{equation*}
		\Ein_{y,a}(x)=-\frac{1}{\ii \kappa} \curl \curl a \Phi(x,y),\qquad \Hin_{y,a}= \curl a \Phi(x,y)
	\end{equation*}
Here $\Phi(x,y)=\Phi_\kappa(x-y)$. We assume that we can place electric dipoles 
with moments given by the Cartesian unit vectors $d_1, d_2$ and $d_3$
at all points $y$ on a sphere $R\Sphere$, $R>\pi$   and measure the corresponding electric fields on the same sphere. Then due to linearity of \eqref{eq:permaxwell} 
we know the total electric field $E_{y,a} (x)$ 
on the sphere $R\Sphere$ for any incident electric dipole field $\Ein_{y,a}$
with moment $a\in \Rset^3$ and  source point $y\in R\Sphere$, and 
the measurements can be arranged in a matrix $w_n(x,y)$ called the \emph{near field} scattering data such that
	\begin{equation*}	
		E_{y,a} (x) = w_n (x,y) a.  
	\end{equation*}
Obviously, $w_n$ is the Green's tensor of the problem. We will further denote by $w_n^\mathrm{s}(x,y)=w_n(x,y)-w_1(x,y)$ the contribution of the scattered field to the matrix, where $w_1(x,y)$ is the near field corresponding to the homogeneous medium case $n\equiv 1$. The problem to reconstruct the refractive index from such measurements can be posed as an operator equation $F_\mathrm{n} (n)= w_n$ where the near field operator is defined as
	\begin{equation*}
		F_\mathrm{n}\colon \solsetb \cap \Hmdom \rightarrow (L^2(R\Sphere \times R\Sphere))^{3 \times 3}, \quad n \mapsto w_n.
	\end{equation*}
	
To describe the second inverse problem, recall that 
	the Silver-M\"uller radiation condition \eqref{eq:silver} implies for the scattered field the asymptotic behavior
	\begin{equation*}
		\Esca_{d,p} (x)= \frac{\ee^{\ii \kappa r}}{r} \rbra{E^\infty_{d,p} (\hat x)+ o (1)}, \quad r\rightarrow \infty
	\end{equation*}
	where $r=\Tabsval x$ and $\hat x = x/r$. If we choose as the incident wave 
a plane wave of the form
	\begin{equation*}
		\Ein_{d,p}(x)= d \times(p\times d) \ee^{\ii \kappa d\cdot x},\qquad \Hin_{d,p}(x)=(\ii \kappa)^{-1} \curl \Ein_{d,p}(x), \qquad x\in\Rset^3
	\end{equation*}
	traveling in direction $d\in \Sphere$ and having polarization $p\in\Cset^3$  we can associate with each refractive index a \emph{far field} $E^\infty$ that maps $(\hat x, d, p)\in \Sphere\times \Sphere \times \Cset^3$ to 
	$E^\infty_{d,p} (\hat x)\in \Cset^3$. As this mapping is linear in $p$, 
	there exists a matrix valued function $e^\infty_n: \Sphere \times \Sphere \rightarrow\Cset^{3\times 3}$ such that
	\begin{equation*}
		 E^\infty_{d,p} (\hat x)=e^\infty_n (\hat x, d) p.
	\end{equation*}
	Thus we can define the far field operator
	\begin{equation}
		F_\mathrm{f}\colon \solsetb \cap \Hmdom \rightarrow (L^2(\Sphere \times \Sphere))^{3\times 3},\quad n \mapsto e^\infty_n. \label{eq:farfieldoperator}
	\end{equation}

\section{Estimation of low order Fourier coefficients}\label{sec:lowfourier}

	A main tool for our proof are \emph{complex geometrical optics (CGO)} solutions, which are solutions to \eqref{eq:permaxwell} having exponential growth in one direction.
	They are important tools in the analysis of scattering problems, for example to prove uniqueness \cite{Calderon1980,Novikov1988,Sylvester1987} or stability estimates \cite{Alessandrini1988,Haehner2001,Isaev2013a,Stefanov1990}. Recently we employed them in \cite{Hohage2015} to prove a variational source condition for acoustic inverse medium scattering problem. Concerning their usage for the electromagnetic case we refer to the constructions for a uniqueness proof in \cite{Colton1992} and for a stability proof in \cite{Haehner2000}, which we will mostly follow. 
We point out that other constructions exist which can also cover the case of non constant magnetic permeability, see \cite{Ola1996,Caro2010}.  For further references about these solutions introduced in \cite{Faddeev1965} we recommend the review \cite{Uhlmann2009}.
	
	Since  the time harmonic Maxwell equations are not of Helmholtz  type, the construction  of CGO solutions is more complicated. Mainly one has to find a transformation of \eqref{eq:permaxwell} such that the resulting equations are of Helmholtz type. The corresponding transformation will be discussed in  Appendix \ref{sec:GOS}, and the following result will be derived:	
	\begin{theorem}\label{thm:GOSofCm}
		 Let $\pi<R$, $n\in\solsetb\cap H^m$ for $m>7/2$ with $\Tnorm{n}{H^m}\leq C_m$, $\kappa>0$ and $\zeta, \eta\in \Cset^3$ such that $\zeta\cdot \zeta=\kappa^2$, $\zeta\cdot \eta=0$ and $\absval{\Im(\zeta)}\geq \tmin$, where
		\begin{equation}\label{eq:GOSofCm:t1def}
			\tmin:=  60 \frac{R}{\pi}(1+\kappa^2) b^{-2} (\Me C_m)^2
		\end{equation}
with the embedding constant $\Me$ given in \eqref{eq:defi_Mm}. 
		Then there exists a solution to \eqref{eq:permaxwell} in the ball $B(2R)$ of the form
		\begin{equation}\label{eq:GOSform}
			E(x,\zeta,\eta)= \ee^{\ii \zeta \cdot x} \sbra{\eta+ f(x,\zeta,\eta) \zeta+ V(x,\zeta,\eta)}, \quad \text{for }x\in B(3R/2)
		\end{equation}
		and $H=(\ii\kappa)^{-1}\curl E$ such that
		\begin{equation*}
			\norm{f(\cdot,\zeta,\eta)}{L^2(B(3R/2))}+\norm{V(\cdot,\zeta,\eta)}{L^2(B(3R/2))}\leq \frac{\Mgos \absval \eta}{\absval{\Im(\zeta)}}.
		\end{equation*}
		with a constant $\Mgos$ depending on $C_m,\kappa, b$ and $R$.
	\end{theorem}

This result is an analog of  \cite[Theorem 4]{Haehner2000}, but we have made 
the dependence of the constant $t_0$ on the parameters explicit. 
	
	\begin{lemma}\label{lem:difftodata}
		Let $R>\pi$, $m>7/2$ and assume that $n_1$ and $n_2$ are two refractive indices satisfying $n_j\in \solsetb \cap H^m$ with $\Tnorm{n_j}{H^m}\leq C_m$ 
		for some $C_m \geq 0$. Let $E_j, H_j \in C^1(B(3R/2))\cap L^2(B(3R/2))$ be solutions to \eqref{eq:permaxwell} in $B(3R/2)$ for $n=n_j$ for $j=1,2$. Then the estimate
		\begin{multline*}
			\absval{\int_{B(\pi)} (n_1-n_2)E_1 E_2\,\dd x}\\
			\leq \Mdata \data \norm{E_1}{L^2(B(3R/2))}\norm{E_2}{L^2(B(3R/2))}
		\end{multline*}
		holds true, where $w_j$ is the near field scattering data for $n=n_j$ for $j=1,2$ and $\Mdata$ depends on $\kappa, R, b$ and $C_m$.
	\end{lemma}
	\begin{proof}
		 The proof follows along the lines of \cite[Lemma 5]{Haehner2000} but other norms are used there therefore we sketch the proof here.
		 For $j=1,2$ extend $(E_j,H_j)$  to radiating solutions $(V_j,W_j)$ fulfilling \eqref{eq:silver}, with $(V_j,W_j)|_{B(3R/2)}=(E_j,H_j)$ and $\nu \times E_j^-=\nu\times V_j^+$ on $3R/2 \Sphere$, where $\nu$ denotes the outer normal vector on $3R/2\Sphere$ and $E_j^-$ and $V_j^+$ denote the inner and outer Dirichlet trace on $3R/2 \Sphere$ of $E_j$ and $V_j$ respectively. By \cite[eq.~(12)]{Haehner2000} the equality 
		 \begin{multline*}
			2\kappa^2 \int_{B(\pi)} (n_1-n_2)E_1 E_2\,\dd x \\
			=\int_{R\Sphere} \sbra{\nu \times \rbra{W_1^{-} -W_1^{+}}\times \nu} \sbra{N_1- N_2} \sbra{\nu \times \rbra{W_2^{-} -W_2^{+}}}\,\dd s
		 \end{multline*}
		 holds true with the operator
		 \begin{equation*}
		 	(N_j a)(x) = 2\nu \times \int_{R\Sphere}  w_j(x,y)  a(y)\,\dd y
		 \end{equation*}
		for $a$ sufficiently smooth and $j=1,2$ (see \cite[Lemma 3 and 10]{Haehner2000} for further properties of $N_j$). Again by \cite[eq.~(11)]{Haehner2000} we know there exists a constant $c$ such that
		\begin{equation*}
			\norm{\nu \times \rbra{W_j^{-} -W_j^{+}}}{L^2(R\Sphere)}\leq c \norm{E_j}{L^2(B(3R/2))}.
		\end{equation*}
Estimating the $L^2$-operator norm of $N_1-N_2$ by the $L^2$-norm 
of its kernel gives the assertion.
	\end{proof}

This result can now be used to derive bounds on Fourier coefficients of 
$n_1-n_2$. Comparing Lemma \ref{lem:fourierdiff} below with 
\cite[Lemma 6]{Haehner2000}, note that our choice of $t\geq \tmin$ yields a 
better control of the difference of the Fourier coefficients in $\varrho$ 
since in our case the first summand is independent of $\varrho$ and in the 
second summand we have a factor of $\varrho$ instead of $\varrho^4$. 
Moreover, the additional factor $\norm{n_1-n_2}{\Hm}$ will be useful. 
That our estimate depends on $\varrho$ is the main difference between this 
result and the comparable result \cite[Lemma 3.3]{Hohage2015} in the acoustic 
case. This will also be the major difference in the proof of the variational 
source condition for the electromagnetic compared to the acoustic case leading 
to the smaller exponent in the index function.
	
	\begin{lemma}\label{lem:fourierdiff}
		Let $R>\pi$, $m>7/2$ and $n_1$ and $n_2$ be two refractive indices such that $n_j\in \solsetb \cap H^m$ with $\Tnorm{n_j}{H^m}\leq C_m$ for some $C_m \geq 0$ with corresponding near field data $w_j$ for $j=1,2$. Let $t\geq \tmin$ with $\tmin$ as in \eqref{eq:GOSofCm:t1def} and $1\leq\varrho\leq 2\sqrt{\kappa^2+t^2}$. Then there exists a constant $\Mdiff$ depending only on $R, \kappa, b$ and $C_m$ such that
		\begin{equation*}
			\absval{\fourierdiff{n_1}{n_2}}\leq \Mdiff \rbra{\data \ee^{3Rt}+\norm{n_1-n_2}{\Hm} \frac{\varrho}{t}}
		\end{equation*}
		holds true for all $\gamma\in \Zset^3$ with $\Tabsval \gamma\leq \varrho$.
	\end{lemma}
	\begin{proof}
		Let $\gamma \in \Zset^3$ with $\Tabsval \gamma \leq \varrho$ be given. Choose $a_1$ and $a_2$ in $\Rset^3$ such that $\{\gamma/\Tabsval \gamma, a_1, a_2\}$ form an orthonormal system of $\Rset^3$. Define for $t\geq \tmin$ the following vectors in $\Cset^3$:
		\begin{equation}
		\begin{aligned}
			&\zeta_1 := -\frac{1}{2} \gamma + \ii t a_1 +\sqrt{\kappa^2+t^2-\frac{\absval{\gamma}^2}{4}} a_2, && \eta_ 1 := \frac{1}{\absval \gamma} \gamma - \ii \frac{\absval \gamma}{2t} a_1,\\
			&\zeta_2 := -\frac{1}{2} \gamma - \ii t a_1 -\sqrt{\kappa^2+t^2-\frac{\absval{\gamma}^2}{4}} a_2, && \eta_ 2 := \frac{1}{\absval \gamma} \gamma + \ii \frac{\absval \gamma}{2t} a_1.
		\end{aligned}\label{eq:fourierdiff:vectors}
		\end{equation}
		These vectors satisfy the relations
		\begin{align*}
			\zeta_1 \cdot \zeta_1=\zeta_2\cdot \zeta_2=\kappa^2, && \zeta_1\cdot \eta_1= \zeta_2\cdot \eta_2=0, && \zeta_1+\zeta_2=-\gamma.
		\end{align*}
		Hence by Theorem \ref{thm:GOSofCm} for $j=1,2$ there exist solutions $E_j, H_j$ to \eqref{eq:permaxwell} in $B(2R)$ with $n$ replaced by $n_j$, 
		which have the form \eqref{eq:GOSform} in $B(3R/2)$ with $\zeta, \eta$ replaced by $\zeta_j, \eta_j$. The product of these solutions has the form:
		\begin{equation}
		\begin{aligned}
			E_1\cdot E_2 &= \ee^{-\ii \gamma \cdot x} \sbra{\eta_1+f_1 \zeta_1+V_1}\cdot \sbra{\eta_2+f_2  \zeta_2+V_2} \\
			&=\ee^{-\ii \gamma \cdot x} \left[1+\frac{\absval{\gamma}^2}{4t^2}- \absval{\gamma} (f_1+f_2) +\eta_1 \cdot V_2 +\eta_2 \cdot V_1 \right. \\
			&\phantom{=\ee^{-\ii \gamma \cdot x}}\quad \left. + f_1 f_2 \bigg( \frac{\absval \gamma^2}{2}-\kappa^2 \bigg) + f_1\zeta_1\cdot V_2+f_2 \zeta_2\cdot V_2+V_1\cdot V_2  \right],
		\end{aligned}\label{eq:fourierdiff:product}
		\end{equation}
		
		By the definition of the Fourier transform together with \eqref{eq:fourierdiff:product} this implies that
		\begin{equation}
		\begin{aligned}
			(2\pi)^{3/2}\absval{\fourierdiff{n_1}{n_2}}\leq& \bigg\lvert\int_{B(\pi)} (n_1-n_2)E_1 E_2\,\dd x \bigg\rvert+ \bigg\lvert\int_{B(\pi)} (n_1-n_2) \\ &\rbra{\frac{\absval{\gamma}^2}{4t^2}- \absval{\gamma} (f_1+f_2) +\eta_1 \cdot V_2 +\eta_2 \cdot V_1} \,\dd x \bigg\rvert \\
			& +\bigg\lvert\int_{B(\pi)} (n_1-n_2) \bigg(f_1 f_2 \Big( \frac{\absval \gamma^2}{2}-\kappa^2 \Big) + f_1\zeta_1\cdot V_2\\
			&+f_2 \zeta_2\cdot V_2+V_1\cdot V_2 \bigg) \, \dd x\bigg\rvert
		\end{aligned}\label{eq:fourierdiff:integralsum}
		\end{equation}
		To estimate further we need the moduli of the vectors in \eqref{eq:fourierdiff:vectors} which are given by
		\begin{align*}
			\absval{\zeta_1}=\absval{\zeta_2}=\sqrt{2t^2+\kappa^2}, && \absval{\eta_1}=\absval{\eta_2}=\sqrt{1+{\absval{\gamma}^2}/{\absval{t}^2}}.
		\end{align*}
		Since $t\geq \tmin \geq \max\{1, \kappa^2\}$ we can estimate that
		\begin{align}\label{eq:rhoovert}
			\frac{\varrho^2}{t^2}\leq \frac{4(t^2+\kappa^2)}{t^2}\leq 8
		\end{align}
		and hence $\Tabsval{\eta_j}\leq 3$ for $j=1,2$.
		
		Applying Lemma \ref{lem:difftodata} to the first integral of the right hand side of \eqref{eq:fourierdiff:integralsum} one obtains
		\begin{multline}
			\absval{\int_{B(\pi)} (n_1-n_2)E_1 E_2\,\dd x}\\ \leq 9 \Mdata \data \ee^{3Rt}\rbra{\norm{1}{L^2(B(3R/2))}+3\Mgos}^2,\label{eq:fourierdiff:1int}
		\end{multline}
		since $t\geq \max\{1, \kappa^2\}$ implies that
		\begin{align*}
			\norm{E_j}{L^2(B(3R/2))}&\leq 3 \ee^{3Rt/2}\rbra{\norm{1}{L^2(B(3R/2))}+\Mgos \frac{\absval{\zeta}+1}{t}}\\
			&\leq 3 \ee^{3Rt/2}\rbra{\norm{1}{L^2(B(3R/2))}+3 \Mgos},\quad \text{since }  \frac{\sqrt{2t^2+\kappa^2}}{t}\leq 2
		\end{align*}
		for $j=1,2$.
		
		The second integral of the right hand side of \eqref{eq:fourierdiff:integralsum} can be estimated by
		\begin{equation}
		\begin{aligned}
			&\absval{\int_{B(\pi)} (n_1-n_2) \rbra{\frac{\absval{\gamma}^2}{4t^2}- \absval{\gamma} (f_1+f_2) +\eta_1 \cdot V_2 +\eta_2 \cdot V_1} \,\dd x }\\
			\leq& \norm{n_1-n_2}{L^2(C(\pi))}\bigg(\frac{1}{\sqrt{2}}\frac{\varrho}{t} \norm{1}{L^2(B(\pi))} +\varrho \rbra{\norm{f_1}{L^2(B(R))}+\norm{f_2}{L^2(B(R))}}\\
			&+3 \rbra{\norm{V_1}{L^2(B(R))}+\norm{V_2}{L^2(B(R))}} \bigg)\\
			\leq&\norm{n_1-n_2}{\Hm} \frac{\varrho}{t} \rbra{\frac{1}{\sqrt{2}}\norm{1}{L^2(B(\pi))} +24\Mgos}
		\end{aligned}\label{eq:fourierdiff:2int}
		\end{equation}
		by the estimate on $f_j$ and $V_j$ in Theorem \ref{thm:GOSofCm} and by \eqref{eq:rhoovert}. Similarly for the third integral
		\begin{equation}
		\begin{aligned}
			&\bigg\lvert\int_{B(\pi)} (n_1-n_2) \bigg(f_1 f_2 \Big( \frac{\absval \gamma^2}{2}-\kappa^2 \Big) + f_1\zeta_1\cdot V_2+f_2 \zeta_2\cdot V_2+V_1\cdot V_2 \bigg) \, \dd x\bigg\rvert\\
			\leq\,&\norm{n_1-n_2}{L^\infty(C(\pi))} \bigg( \Big( \frac{\absval \varrho^2}{2}+\kappa^2 \Big)\norm{f_1}{L^2(B(R))}\norm{f_2}{L^2(B(R))}+ \\
			& \sqrt{2t^2+\kappa^2}\Big( \norm{f_1}{L^2(B(R))}\norm{V_2}{L^2(B(R))}+\norm{V_1}{L^2(B(R))}\norm{f_2}{L^2(B(R))}\Big)\\
			& +\norm{V_1}{L^2(B(R))}\norm{V_2}{L^2(B(R))}\bigg)\\
			\leq\,& \Me \norm{n_1-n_2}{\Hm} \rbra{18 \Mgos^2 \frac{\varrho}{t}+36 \Mgos^2 \frac{1}{t}+9\Mgos^2 \frac{1}{t^2}}\\
			\leq\,& 63 \Me \Mgos^2 \norm{n_1-n_2}{\Hm} \frac{\varrho}{t},
		\end{aligned}\label{eq:fourierdiff:3int}
		\end{equation}
		where we used in addition that $t\geq 60\kappa^2$ implies
		\begin{align}\label{eq:t_rho_kappa}
			\frac{1}{t} \rbra{\frac{\varrho}{2}+\kappa^2}\leq 2.
		\end{align}

		Combining \eqref{eq:fourierdiff:integralsum} to \eqref{eq:fourierdiff:3int} one sees, that there exists a constant $\Mdiff$ such that
		\begin{equation*}
			\absval{\fourierdiff{n_1}{n_2}}\leq \Mdiff \rbra{\data \ee^{3Rt}+\norm{n_1-n_2}{\Hm} \frac{\varrho}{t}},
		\end{equation*}
		where $\Mdiff$ depends only on $R, \kappa, b$ and $C_m$.
	\end{proof}

\section{Proof of Theorem \ref{thm:vscnear}}\label{sec:proofnear}
	Before we proof the main theorem of this paper we will rewrite the variational source condition \eqref{eq:vsc} in an equivalent way	
	\begin{equation}\label{eq:resultproduct}
		\Re \pairing{\ntrue}{\ntrue-n}_{\Xspace}\leq \frac{1-\beta}{2} \norm{n-\ntrue}{\Xspace}^2+\psi \rbra{\norm{F(n)-F(\ntrue)}{\Yspace}^2},
	\end{equation}
	since this form is more convenient for the proof. 
	
	\begin{proof}[Proof of Theorem \ref{thm:vscnear}]
		Let $n\in \solsetb\cap \Hm$ be given. To be able to apply the previous lemmata we need that $\Tnorm{n}{H^m}\leq C_m$ for some $C_m>0$. Thus consider the case $\Tnorm{n-\ntrue}{\Hm}>4C_s$. Then by applying Cauchy-Schwarz we obtain
		\begin{align*}
			\Re \pairing{\ntrue}{\ntrue-n}_{\Hm}\leq \norm{\ntrue}{\Hm} \norm{\ntrue-n}{\Hm}\leq \frac{1}{4}\norm{n-\ntrue}{\Hm}^2
		\end{align*}
		which implies \eqref{eq:resultproduct}.
		
		Hence, it remains to treat the case $\Tnorm{n}{H^m}\leq5 C_s$. 
		Let us introduce the notation $\mult:=1+\Tabsval{\gamma}^2$ for $\gamma\in\Zset^3$.
		We choose $t\geq \tmin$ with $C_m=5C_s$ in Lemma \ref{lem:fourierdiff} and $1\leq \varrho \leq 2 \sqrt{\kappa^2+t^2}$ and set $\delta=\norm{w_{n}-w_{\ntrue}}{(L^2(R\Sphere\times R\Sphere))^{3\times 3}}$ to obtain
		\begin{equation}
		\begin{aligned}
			\Re\pairing{\ntrue}{P_\varrho(\ntrue-n)}_{\Hm}\!\! &= \Re\!\!\! \sum_{\gamma \in \Zset^3\cap B(\varrho)}\!\!\!\!\! \mult^m \fourier[\ntrue] \overline{\fourierdiff{\ntrue}{n}}\\
			&\leq \Mdiff\rbra{\delta \ee^{3Rt}+\norm{\ntrue-n}{\Hm} \frac{\varrho}{t}}\!\!\!\!\!\sum_{\gamma \in \Zset^3\cap B(\varrho)}\!\!\!\!\! \mult^m \absval{\fourier[\ntrue]}, 
		\end{aligned}\label{eq:vscnear:small}
		\end{equation}
		where $P_\varrho(f)$ is the projection onto the Fourier coefficients of $f$ with $\Tabsval \gamma\leq \varrho$. By Lemma 4.3 of \cite{Hohage2015} there exists a constant $\Mms$ depending on $m$ and $s$ such that the remaining sum on the right hand side of \eqref{eq:vscnear:small} can be bounded by
		\begin{equation}
			\sum_{\gamma \in \Zset^3\cap B(\varrho)}\!\!\!\!\! \absval{\mult^m \fourier[\ntrue]}\leq \sqrt{\sum_{\gamma \in \Zset^3} \mult^s \Tabsval{\fourier[\ntrue]}^2} \sqrt{\sum_{\gamma \in \Zset^3\cap B(\varrho)}\!\!\!\!\! \mult^{2m-s}}\leq \Mms C_s \varrho^\tau \label{eq:vscnear:sum}
		\end{equation}
		with $\tau=\max\{2m+3/2-s,0\}$ for $s\neq 2m+3/2$.
		
		To obtain a bound on the high frequencies we use the higher smoothness of $\ntrue$ similar to \cite{Haehner2001,Hohage2015,Isaev2014}. Therefore, we first apply Cauchy-Schwarz and then Young's inequality to obtain
		\begin{equation}
		\begin{aligned}
			&\phantom{\leq}\Re\pairing{\ntrue}{(I-P_\varrho)(\ntrue-n)}_{\Hm}\!\! = \Re\!\!\!\sum_{\gamma \in \Zset^3\setminus B(\varrho)}\!\!\!\!\! \mult^m \fourier[\ntrue] \overline{\fourierdiff{\ntrue}{n}}\\
			&\leq \sqrt{\sum_{\gamma \in \Zset^3\setminus B(\varrho)}\!\!\!\!\! \mult^m \absval{\fourier[\ntrue]}^2 } \sqrt{\sum_{\gamma \in \Zset^3\setminus B(\varrho)}\!\!\!\!\! \mult^m \absval{\fourierdiff{\ntrue}{n}}^2}\\
			&\leq 2\!\!\!\sum_{\gamma \in \Zset^3\setminus B(\varrho)}\!\!\!\!\! \mult^m \absval{\fourier[\ntrue]}^2 + \frac{1}{8} \norm{\ntrue-n}{\Hm}^2.
		\end{aligned}\label{eq:vscnear:high}
		\end{equation}
		The smoothness assumption on $\ntrue$ now implies that
		\begin{equation}
			\sum_{\gamma \in \Zset^3\setminus B(\varrho)}\!\!\!\!\! \mult^m \absval{\fourier[\ntrue]}^2\leq \frac{1}{(1+\varrho^2)^{s-m}} \sum_{\gamma \in \Zset^3\setminus B(\varrho)}\!\!\!\!\! \mult^s \absval{\fourier[\ntrue]}^2 \leq \varrho^{2(m-s)} C_s^2.\label{eq:vscnear:smooth}
		\end{equation}
		
		Combining \eqref{eq:vscnear:small} to \eqref{eq:vscnear:smooth} we arrive at
		\begin{equation}
		\begin{aligned}
			\Re\pairing{\ntrue}{\ntrue-n}_{\Hm}\!\! &\leq \frac{1}{8} \norm{\ntrue-n}{\Hm}^2  + 2 \varrho^{2(m-s)} C_s^2\\
			&\phantom{leq} + \Mdiff\Mms C_s \varrho^\tau \rbra{\delta \ee^{3Rt}+\norm{\ntrue-n}{\Hm} \frac{\varrho}{t}}\\
			&\leq \rbra{\frac{1}{8}+\frac{1}{8} \frac{\varrho^{2(1+\tau+s-m)}}{\varepsilon t^2}} \norm{\ntrue-n}{\Hm}^2\\
			&\phantom{leq} + 2 C_s^2 (1+\varepsilon \Mdiff^2\Mms^2) \varrho^{2(m-s)}+\Mdiff\Mms C_s \varrho^\tau \delta \ee^{3Rt}
		\end{aligned}\label{eq:vscnear:trho}
		\end{equation}
		for all $\varepsilon>0$ by Young's inequality.
		
		Next we choose the free parameters $\varrho, t$ and $\varepsilon$ in dependence of $\delta>0$ in such a way that the right hand side of \eqref{eq:vscnear:trho} is approximately minimal and the constraints of Lemma \ref{lem:fourierdiff} are satisfied. We choose
		\begin{equation}
			\varepsilon = (9R)^2, \qquad 9Rt = \ln (3+\delta^{-2}) = \varrho^{1+\tau+s-m}. \label{eq:vscnear:parameter}
		\end{equation}
		Then the constraint $\varrho\geq 1$ is fulfilled. Since $1+\tau+s-m\geq m+5/2>6$,  
		there exists a $t^\prime$ such that $2 t\geq (9Rt)^{1/(1+\tau+s-m)}$ for all 
		$t\geq t^\prime$. Hence by strengthening the constraint on $t$ to 
		$t\geq \tmin^\prime:=\max\{\tmin,t^\prime\}$ the upper bound on $\varrho$ is 
		satisfied since then $2\sqrt{\kappa^2+t^2}>2t\geq (9Rt)^{1/(1+\tau+s-m)}=\varrho$ 
		(see \eqref{eq:t_rho_kappa} for the first inequality). 
		However $t \geq t^\prime$ is only satisfied for $\delta\leq \dmax$, where 
		$\dmax:=(\exp(9R \tmin^\prime)-3)^{-1/2}$ (or $\dmax=\infty$ if 
		$\exp(9R \tmin^\prime)\leq 3$), hence the case $\delta>\dmax$ has to be handled 
		subsequently. For $\delta\leq\dmax$ plugging our choice 
		\eqref{eq:vscnear:parameter} into \eqref{eq:vscnear:trho} yields
		\begin{align*}
			\Re\pairing{\ntrue}{\ntrue-n}_{\Hm}\!\! 
			&\leq \frac{1}{4} \norm{\ntrue-n}{\Hm}^2
			+2 C_s^2 (1+(9R)^2 \Mdiff^2\Mms^2) \rbra{\ln (3+\delta^{-2})}^{-2\nu}\\
			&\phantom{leq} +\Mdiff\Mms C_s \rbra{\ln (3+\delta^{-2})}^{\lambda} \delta (3+\delta^{-2})^{1/3}
		\end{align*}
		with 
		\begin{equation*}
			\nu:=\min \cbra{\frac{s-m}{m+5/2},\frac{s-m}{s-m+1}}, \qquad \lambda=\max\cbra{0, \frac{2m+3/2-s}{m+5/2}}.
		\end{equation*}
		Since the term in the second line tends faster to $0$ then the last one in the first for $\delta\searrow 0$ there exists some constant $\tilde A$ such that for $\delta\leq \dmax$
		\begin{equation*}
			\Re\pairing{\ntrue}{\ntrue-n}_{\Hm}\!\! \leq \frac{1}{4} \norm{\ntrue-n}{\Hm}^2+\tilde A \rbra{\ln (3+\delta^{-2})}^{-2\nu}
		\end{equation*}
		showing \eqref{eq:resultproduct} for $\delta \leq \dmax$.
		
		If on the other hand $\delta>\dmax$ applying again Cauchy-Schwarz and Young's inequality yields
		\begin{equation}\label{eq:vscnear:alldelta}
			\Re\pairing{\ntrue}{\ntrue-n}_{\Hm}\!\! \leq \frac{1}{4} \norm{\ntrue-n}{\Hm}^2+C_s^2.
		\end{equation}
		Therefore, Theorem \ref{thm:vscnear} holds true with $A=\max\{\tilde A, C_s^2 (\ln (3+\dmax^{-2}))^{2\nu}\}$.
	\end{proof}

\section{From near to far field data}\label{sec:neartofar}

	In this section we show that the difference of the near field measurements for two refractive indices can be bounded by a function of their far field measurements. The idea of the proof of Theorem \ref{thm:vscfar} is then to insert this bound into the variational source condition for the near field data.
	
	Since both kinds of data are measured on spheres we will express them in terms of a series representation using spherical harmonics. Let $Y_l^k$ for $l\in \Nset_0$ and $k\in \Zset, \Tabsval k \leq l$ denote the spherical harmonics and set for convenience
	\begin{equation*}
		M:=\cbra{(l_1,k_1,l_2,k_2)\in \Nset_0\times\Zset\times\Nset_0\times\Zset \colon \absval{k_1}\leq l_1,\absval{k_2}\leq l_2}.
	\end{equation*}
	Then we will define the (matrix valued) Fourier coefficients of the far field by
	\begin{equation*}
		\alpha(l_1,m_1,l_2,m_2) = \int_{\Sphere} \int_{\Sphere} e^\infty_n(\hat x, d) \overline{Y^{k_2}_{l_2}(\hat x)} \overline{Y^{k_1}_{l_1}(d)} \, \dd \hat x\, \dd d, \quad (l_1,m_1,l_2,m_2)\in M.
	\end{equation*}
	Denoting by $\smash{h^{(1)}_l}$ the spherical Hankel function of first kind of order $l$ now gives us a series representation for $w_n^\mathrm{s}$ in terms of the Fourier coefficients of the far field. A short introduction on the functions $Y_l^k$ and $\smash{h^{(1)}_l}$ can be found in \cite{Colton2013}.
	
	\begin{lemma}
		Let $R>\pi$ and $n\in \solsetb\cap \Hm$ for $m>7/2$ be a refractive index with $e^\infty_n$ and $w_n$ the corresponding far and near field. Denote by $\alpha(l_1,m_1,l_2,m_2)$ the Fourier coefficients of $e^\infty_n$. Then the scattered part of the near field data has the representation
		\begin{equation*}
		\begin{aligned}
			w_n^\mathrm{s}(x,y) = -\frac{\kappa^4}{4\pi} \sum_{(l_1,k_1,l_2,k_2)\in M}& \ii^{l_1-l_2} \alpha(l_1,k_1,l_2,k_2)\, h_{l_1}^{(1)} (\kappa R)\, h_{l_2}^{(1)} (\kappa R)\\& Y^{k_1}_{l_1}\!\!\rbra{\frac{x}{\absval x}}\, Y^{k_2}_{l_2}\!\!\rbra{\frac{y}{\absval y}}, \qquad |x|>|y|=R.
		\end{aligned}
		\end{equation*}
		Furthermore the series converges absolutely and uniformly on compact sets.
	\end{lemma}
	\begin{proof}
		See \cite[Lemma 10]{Haehner2000}.
	\end{proof}
	
	A similar result for the acoustic case has been derived by Stefanov in \cite{Stefanov1990}. In \cite{Haehner2001} this results was used together with stability estimates for the inverse of compact linear operators under spectral source conditions to estimate the difference of near field data by the difference of far field data. Since the series representation of the near field for the acoustic and the electromagnetic case coincide up to the fact that the electromagnetic Fourier coefficients are matrix valued, one obtains the following lemma along the 
	lines of \cite{Haehner2001}:
	\begin{lemma}\label{lem:neartofar}
		Let $R>\pi$, $m>7/2$, $C_m>0$ and $0<\theta<1$. Then there exist constants $\omega,\varrho,\dmax>0$ such that for any two refractive indices 
		$n_1,n_2\in \solsetb\cap H^s$ with $\Tnorm{n_j}{H^s}\leq C_s$ for some 
		$C_s>0$, we have 
		\begin{equation*}
			\norm{w_2-w_1}{(L^2(2R\Sphere \times 2R\Sphere))^{3 \times 3}}^2 \leq  \varrho^2 \exp \rbra{- \rbra{- \ln \frac{\norm{u^\infty_2-u^\infty_1}{(L^2(\Sphere \times \Sphere))^{3 \times 3}}}{\omega \varrho}}^\theta } 
		\end{equation*}
		if $\Tnorm{u^\infty_2-u^\infty_1}{(L^2(\Sphere \times \Sphere))^{3 \times 3}}\leq \dmax$  where $w_j$ and $u^\infty_j$ denote near and far field scattering data for $n_j$,  $j=1,2$. 
	\end{lemma}
	
\section{Proof of Theorem \ref{thm:vscfar}}\label{sec:prooffar}
	Having a variational source condition for near field data and a way to bound near 
field data by far field data, the proof of Theorem \ref{thm:vscfar} now proceeds 
similar to the acoustic case.
	\begin{proof}[Proof of Theorem \ref{thm:vscfar}]
		The case $\Tnorm{\ntrue-n}{\Hm} > 4C_s$ can be treated as in the proof of Theorem \ref{thm:vscnear}.
		
		For $\Tnorm{n}{\Hm}\leq 5 C_s$ we can apply Lemma \ref{lem:neartofar}. Setting
		\begin{equation*}
			\delta:= \Tnorm{F_\mathrm{f}(\ntrue)- F_\mathrm{f}(n)}{(L^2(\Sphere \times \Sphere))^{3 \times 3}}, \quad \varphi(t):= \varrho^2\exp(-(-\ln(\sqrt{t})+\ln(\omega\varrho))^{\theta}),
		\end{equation*}
		it follows from Theorem \ref{thm:vscnear} and the monotonicity of $\psi_\mathrm{n}$ that the variational source condition \eqref{eq:vsc} holds true with $\psi(t) = \psi_\mathrm{n}(\varphi(t))$ if $\delta\leq \dmax$. Bounding $\psi_\mathrm{n}(t)\leq A(\ln t^{-1})^{-2\nu}$ for $t<1$ we obtain 
		\begin{equation*}
			\psi_\mathrm{n}(\varphi(t)) \leq  A\left(-\left(-\ln(\sqrt{t})+\ln (\varrho\omega)\right)^{\theta}-\ln\varrho^2\right)^{-2\nu} \quad \text{for } \sqrt{t}\leq\min\left\{\dmax,\frac{1}{2}\right\}.
		\end{equation*}
		Hence, it is easy to see that there are constants $B>0$ and $\dmaxtilde \in (0,\min\{\dmax,\frac{1}{2}\}]$ such that 
		\begin{equation*}
			\psi_\mathrm{n}(\varphi(t))\leq  B(\ln(3+t^{-1}))^{-2\nu\theta}\quad \text{for }\sqrt{t}\leq \dmaxtilde.
		\end{equation*}
		This shows  \eqref{eq:vsc} for $\delta\leq \dmaxtilde$.
		
		The case $\delta>\dmaxtilde$ is again treated as in the proof of Theorem \ref{thm:vscnear}, see \eqref{eq:vscnear:alldelta}.
	\end{proof}

\appendix

\section{Complex geometrical optics solutions for electromagnetic inverse scattering}\label{sec:GOS}
	In the following we sketch the construction of CGO solution for electromagnetic scattering. We will mostly follow the arguments in \cite{Haehner2000} (a more detailed version by the same author can be found in \cite{Haehner1998}), 
which is based on an idea in \cite{Colton1992}. However, we will make the lower bound on $\Tabsval{\Im(\zeta)}$ more explicit which will enable us to derive a better value of the exponent $\nu$ for fixed values of $m$ and $s$, so the limit for $s\rightarrow \infty$ is the same.
	
	In this section we will always assume that $\pi<\Rp<\Rpp$, $n\in\solsetb\cap\Hm$ for $m>7/2$, $\kappa>0$ and $\zeta, \eta\in \Cset^3$ such that $\zeta\cdot \zeta=\kappa^2$ and $\zeta\cdot \eta=0$.
	
	Our goal is to construct solutions to the Maxwell equations with an electric field of the form $E(x)=\ee^{\ii\zeta\cdot x}(\eta +r(x,\zeta,\eta))$ where the term $r$ decompses into a bounded and a decaying part as $\Tabsval{\zeta}\rightarrow \infty$. In order to do so one needs to constructs an (unphysical) fundamental solution $\Psi_\zeta$ to the Helmholtz equation in a ball around the origin. A construction using periodic spaces was developed in \cite{Haehner1996} and can also be found in the textbook \cite{Colton2013}. One obtains that the corresponding volume integral operator $G_\zeta$ with kernel $\ee^{-\ii \zeta(x-y)}\Psi_\zeta(x-y)$ fulfills the estimate  
	\begin{equation}\label{eq:Gzeta_estim}
	\Tnorm{G_\zeta f}{L^2(B(\Rp))}\leq \Rpp/(\pi \Im (\zeta))\Tnorm{f}{L^2(B(\Rp))}
	\end{equation}
	and hence is a contraction for $\Tabsval{\Im(\zeta)}$ large enough.
	
	Similar to the construction of CGO solutions in the acoustic case, one can now use the Lippmann-Schwinger equation where one replaces the usual fundamental solution by the unphysical $\Psi_\zeta$ and uses $\eta \ee^{\ii\zeta\cdot x}$ as an incident field to construct a solution to \eqref{eq:permaxwell}:
	\begin{lemma}\label{lem:lipschwGOS}
		Suppose $\Tabsval{\Im(\zeta)}\geq 
		2\kappa^2 (\Rpp/\pi) \Tnorm{1-n}{L^\infty(\Rset^d)}+1$. 
		Set $\Ein=\eta \ee^{\ii \zeta \cdot x}$ and $\Hin=(\ii\kappa)^{-1}\curl\Ein$. 
		Let $E\in C(\overline{B(R)})$ be a solution to
		\begin{align}\label{eq:lipschwGOS}
		\begin{aligned}
	E(x)&= \Ein(x) - \kappa^2 \int_{B(\pi)} \Psi_\zeta(x-y) \rbra{1-n(y)} E(y)\, \dd y\\ &+ \grad \int_{B(\pi)} \Psi_\zeta(x-y) \frac{1}{n(y)} \grad n(y) \cdot E(y) \,\dd y, \qquad \ x\in \overline{B (\Rp)}. 
		\end{aligned}
		\end{align}
		Then $E\in C^2(B(\Rp)))$, and $E$ and $H:=(i\kappa)^{-1}\curl E$ satisfy the perturbed Maxwell equation \eqref{eq:permaxwell} in $B(\Rp)$.
	\end{lemma}
	\begin{proof}
		See \cite[Lemma 13]{Haehner2000}. 
	\end{proof}
	
	To show uniqueness and the form $E=\ee^{\ii\zeta\cdot x}(\eta +r)$ of 
	solutions to \eqref{eq:lipschwGOS}, one needs a further characterization of these solution in the form of a Helmholtz type equation. 
	In \cite{Colton1992} a matrix-valued function $\Q\colon\Rset^3\rightarrow \Cset^{6\times 6}$ with the same support as $n$ 	was introduced such that 
if $(E,H)$ fulfills \eqref{eq:permaxwell}, then the field $(\tilde E,\tilde H)=(n^{1/2}E,H)$ fulfills the Helmholtz equation
	\begin{equation*}
		(\Delta+\kappa^2) \begin{pmatrix} \tilde E \\ \tilde H\end{pmatrix}= \Q\begin{pmatrix} \tilde E \\ \tilde H\end{pmatrix}
		\qquad \text{where}\quad
		\rbra{\Q \begin{pmatrix} A\\ B \end{pmatrix}}(x) 
:=\Q(x) \begin{pmatrix} A(x)\\B(x)\end{pmatrix}  
	\end{equation*}	
	for $x\in\Rset^3$. The matrix $\Q$ is defined such that 	
	\begin{multline}
		\Q \begin{pmatrix} A \\B\end{pmatrix}= 
		\begin{pmatrix} \kappa^2 (1\!-\!n) A - \ii \kappa n^{-1/2} \grad n \times B -\rbra{A\!\cdot\! \grad}\rbra{\frac{1}{n} \grad n}+\rbra{n^{-1/2} \Delta n^{1/2}}A \\ 
		\kappa^2 (1-n) B+\ii \kappa n^{-1/2} \grad n \times A  \end{pmatrix}\nonumber
	\end{multline}
or more explicitly	
\begin{align*}
			\Q&=\kappa^2 (1-n)  1_6 
	+\frac{\ii\kappa}{\sqrt{n}} \begin{pmatrix}  0_3& -\nabla n\times \\ \nabla n \times &  0_3 \end{pmatrix} 
			+\begin{pmatrix} -D\rbra{\frac{\grad n}{n}} + (n^{-1/2}\Delta n^{1/2})1_3 & 0_3\\ 0_3 & 0_3\end{pmatrix} 
		\end{align*}
		where $1_k$ and $0_k$ denote the $k\times k$ unit and zero matrix respectively, 
		$D(V)$ the Jacobian of a vector field $V$, and
		\begin{equation*}
			\rbra{\nabla n \times} = \begin{pmatrix}0 & -\diffq{n}{z} & \diffq{n}{y} \\ \diffq{n}{z} & 0 & -\diffq{n}{x} \\ -\diffq{n}{y} & \diffq{n}{x} & 0\end{pmatrix}.
		\end{equation*}

\begin{lemma}\label{lem:Qbound}
Let $n\in \solsetb\cap H^m$ with $m>7/2$ and $\Tnorm{n}{H^m}\leq C_m$, 
and let $\Me$ be given by \eqref{eq:defi_Mm}. Then  
\begin{equation*}
\|\Q(x)\|_{2}\leq 15 (1+\kappa^2)b^{-2}(\Me C_m)^2\qquad \mbox{for all } x\in B(\pi). 
\end{equation*}
\end{lemma}

\begin{proof}
Estimating summand by summand, bounding the $\|\cdot\|_2$ norm of non-diagonal matrices by the Frobenius norm,   and using $n^{-1/2}\Delta n^{1/2} 
= \frac{1}{2}n^{-1}\Delta n -\frac{1}{4}n^{-2}|\nabla n|^2 $
we obtain 
		\begin{align*}
	\norm{\Q(x)}{2}\leq\,& 
	    \kappa^2 (1+\Me C_m) + 2\kappa\sqrt{3} b^{-1/2} (\Me C_m)\\
			&+ 3 \rbra{b^{-2} (\Me C_m)^2+b^{-1} (\Me C_m)}\\
			&+\frac{1}{4} b^{-2} \rbra{\Me C_m}^2+\frac{1}{2}b^{-1} (\Me C_m)\\
			\leq\, & (4\sqrt 3 +6+ \frac{3}{4}) (1+\kappa^2) b^{-2} (\Me C_m)^2\\
			\leq\, & 15 (1+\kappa^2)b^{-2}(\Me C_m)^2
		\end{align*}
		where we have used that $\Me C_m\geq 1$ and $\low\geq 1$ due to $n(x)=1$ for $\absval x \geq \pi$. 
\end{proof}

\begin{lemma}\label{lem:primeGOS}
		Let the assumptions of Lemma \ref{lem:lipschwGOS} be fulfilled and let $E$ be the solution to \eqref{eq:lipschwGOS}. Define $E^\prime(x)=\ee^{-\ii\zeta \cdot x} n^{1/2}(x) (E- \Ein)(x)$ and $H^\prime(x)=\ee^{-\ii\zeta \cdot x}  (H- \Hin)(x)$ for $x\in \overline{B(\Rp)}$. Then
		\begin{equation}
			\begin{pmatrix} E^\prime\\H^\prime\end{pmatrix} + G_\zeta \Q \begin{pmatrix} E^\prime\\H^\prime\end{pmatrix}
			=\begin{pmatrix} F_1(\cdot,\zeta,\eta)\\F_2(\cdot,\zeta,\eta)\end{pmatrix}\label{eq:primeGOS}
		\end{equation}
		where
		\begin{equation*}
			\begin{pmatrix} F_1(\cdot,\zeta,\eta)\\F_2(\cdot,\zeta,\eta)\end{pmatrix}:= -G_\zeta \begin{pmatrix} (-\ii n^{-1/2}\zeta\cdot\grad n-\Delta n^{1/2}) \eta\\ 0\end{pmatrix} - G_\zeta  \Q \begin{pmatrix} n^{1/2} \eta \\ \kappa^{-1} \zeta \times \eta \end{pmatrix}.
		\end{equation*}
	\end{lemma}
	\begin{proof}
		See \cite[Lemma 14]{Haehner2000}.
	\end{proof}
	
There is only a sketch of the main theorem of existence of CGO solutions in 
\cite{Haehner2000} and also the dependence of the constants $\tmin$ and $\Mgos$  
appearing below on the other constants is not specified. This has been done more 
explicitly in \cite{Haehner1998}. For convenience we include the proof here. 
\begin{proposition}\label{prop:GOS}
		Suppose that $n\in\solsetb\cap\Hm$  with $\Tnorm{n}{\Hm}\leq C_m$ and  
\begin{equation}\label{eq:Im_zeta}
\Tabsval{\Im (\zeta)}\geq 
\max\left\{2\frac{\Rpp}{\pi}\max_{x\in\overline{B(\pi)}} \norm{\Q(x)}{2},
2\kappa^2 \frac{\Rpp}{\pi} \Tnorm{1-n}{L^\infty(\Rset^d)}+1\right\}.
\end{equation} 
Then both the equations \eqref{eq:lipschwGOS} and \eqref{eq:primeGOS} have 
unique solutions. Moreover, 
\begin{equation}\label{eq:bound_primeGOS}
\norm{\begin{pmatrix}E^{\prime}\\H^{\prime}\end{pmatrix}}{L^2(B(R'))}\leq 
2\norm{\begin{pmatrix}F_1(\cdot,\zeta,\eta)\\F_2(\cdot,\zeta,\eta)\end{pmatrix}}{L^2(B(R'))}.
\end{equation}
\end{proposition}
	\begin{proof}
		If $\zeta$ satisfies \eqref{eq:Im_zeta}, it follows from \eqref{eq:Gzeta_estim} that 
		\begin{equation*}
		\norm{G_\zeta \Q \begin{pmatrix} E^\prime \\ H^\prime \end{pmatrix}}{L^2(B(\Rp))}
		\!\!\!\!\!\leq \frac{\Rpp \sup_x\norm{\Q(x)}{2}}{\pi \absval{\Im(\zeta)}} \norm{\begin{pmatrix} E^\prime \\ H^\prime \end{pmatrix}}{L^2(B(\Rp))}
		\!\!\!\!\!\leq \frac{1}{2} \norm{\begin{pmatrix} E^\prime \\ H^\prime \end{pmatrix}}{L^2(B(\Rp))}
		\end{equation*}
		for all $(E^\prime,H^\prime)\in L^2(B(\Rp))^6$. Hence $\|G_{\zeta}\Q\|\leq 1/2$, 
		and by the Neumann series eq.~\eqref{eq:primeGOS} has a unique solution 
		for all right hand sides. Moreover, $\|(I+G_{\zeta}\Q)^{-1}\|\leq 2$,  
		which implies the bound \eqref{eq:bound_primeGOS}. 

Since by Lemma \ref{lem:primeGOS} for every solution $E$ to the homogeneous equation \eqref{eq:lipschwGOS} the function $(\ee^{-\ii \zeta \cdot x}$ $ n^{1/2} E,\ee^{-\ii \zeta \cdot x}(\ii\kappa)^{-1}\curl E)$ yields a solution to the homogeneous equation \eqref{eq:primeGOS}, the solution to \eqref{eq:lipschwGOS} is unique. 
Since the integral operator in \eqref{eq:lipschwGOS} is compact, existence 
of solutions to \eqref{eq:lipschwGOS} now follows from Riesz-Fredholm theory. 
\end{proof}

From Proposition \ref{prop:GOS} it is not yet clear how large we have to choose 
$\Im(\zeta)$ in terms of our set of parameters $C_m, b, \kappa$ and $R$. 
We will show that the explicit condition \eqref{eq:GOSofCm:t1def} in Theorem 
\ref{thm:GOSofCm} implies 
\eqref{eq:Im_zeta}. Note that our choice of $\Rp$ and $\Rpp$ is arbitrary, 
as long as $\Rpp>\Rp>R$, but their values influence the choice of other parameters.

	\begin{proof}[Proof of Theorem \ref{thm:GOSofCm}]
	We first note from Lemma \ref{lem:Qbound} that for $\Rp=\frac{3}{2}R$ and $\Rpp=2R$ 
	the condition $\Im(\zeta)\geq \tmin$ with $\tmin$ defined in 
	\eqref{eq:GOSofCm:t1def} implies \eqref{eq:Im_zeta} since 
	$\Me C_m\geq \|n\|_{L^\infty}\geq 1$. 


As $|\zeta|^2 = |\Re(\zeta)|^2 + |\Im(\zeta)|^2 = 2|\Im(\zeta)|^2 + \kappa^2$, 
we can bound $|\zeta|$ in terms of $|\Im(\zeta)|$, and we obtain using 
the definition of $(F_1(\cdot,\zeta,\eta), F_2(\cdot,\zeta,\eta))$, 
and	\eqref{eq:Gzeta_estim} that there exists a constant $\tilde M$ depending on $C_m, b, \kappa$ and $\Rpp$ such that
		\begin{equation*}
			\norm{\begin{pmatrix}F_1(\cdot,\zeta,\eta)\\ F_2(\cdot,\zeta,\eta)\end{pmatrix}}{L^2(B(\Rp))}\leq \tilde M \absval{\eta}.
		\end{equation*}
		Therefore, 
		$\Tnorm{(E^\prime,H^\prime)}{L^2(B(\Rp))}\leq 2 \tilde M \Tabsval \eta$ 
		by \eqref{eq:bound_primeGOS}. 
		Writing down the equation for $E^\prime$ explicitly one obtains
		\begin{equation*}
		\begin{aligned}
			E^\prime =& -G_\zeta\rbra{-\ii n^{-1/2}\rbra{\zeta\cdot\grad n \eta+\grad n \times(\zeta \times \eta)}}\\
			&- G_\zeta\Big( \kappa^2 (1-n) \tilde E^\prime - \ii \kappa n^{-1/2}\grad n \times H^\prime - (\tilde E^\prime\cdot \grad) (\frac{1}{n}\grad n)\\
			&\phantom{- G_\zeta\Big(}+(n^{-1/2}\Delta n^{1/2})\tilde E^\prime-\Delta n^{1/2} \eta \Big),
		\end{aligned}
		\end{equation*}
		where $\tilde E^\prime = E^\prime+n^{1/2} \eta$. Using the vector identity $a\times (b\times c)= b(a\cdot c)-c(a\cdot b)$ in the first line one obtains
		\begin{equation*}
			E^\prime=G_\zeta\rbra{\ii n^{-1/2}\grad n\cdot \eta}\zeta+V^\prime(\cdot,\zeta,\eta)
		\end{equation*}
		with $V^\prime$ denoting the second summand of the previous equation. Since $ E^\prime=\ee^{-\ii \zeta\cdot x } n^{1/2}(E- \Ein)$ this shows that $E$ has the claimed form with
		\begin{align*}
			f(\cdot,\zeta,\eta) := n^{-1/2} G_\zeta \rbra{\ii n^{-1/2}\grad n\cdot \eta},
\qquad			V(\cdot,\zeta,\eta) := n^{-1/2} V^\prime(\cdot,\zeta,\eta),
		\end{align*}
		and the desired estimate follows from the bound \eqref{eq:Gzeta_estim}.
\end{proof}

\section*{Acknowledgments} Financial support through CRC 755, project C09 is 
gratefully acknowledged. 

\bibliographystyle{amsxport}
\bibliography{em_lit.bib}

\end{document}